\documentclass[a4paper, leqno]{amsart}
\usepackage{amsmath}
\usepackage{amsthm}
\usepackage{amssymb, latexsym}
\usepackage{enumerate}
\usepackage[utf8]{inputenc}

\usepackage[
      colorlinks=true,    
      urlcolor=blue,    
      menucolor=blue,    
      linkcolor=blue,    
      pagecolor=blue,    
      bookmarks=true,    
      bookmarksopen=true,    
      hyperfootnotes=false,    
      pdfpagemode=UseOutlines    
]{hyperref}

\newtheorem{theorem}{Theorem}
\newtheorem{lemma}[theorem]{Lemma}
\newtheorem{proposition}[theorem]{Proposition}

\newcommand{\R}{\mathbb{R}}
\newcommand{\Z}{\mathbb{Z}}
\newcommand{\Q}{\mathbb{Q}}
\newcommand{\C}{\mathbb{C}}

\newcommand{\p}{\mathfrak{p}}
\renewcommand{\P}{\mathfrak{P}}

\renewcommand{\d}{\mathfrak{d}}
\newcommand{\D}{\mathfrak{D}}
\renewcommand{\O}{\mathcal{O}}
\newcommand{\f}{\mathfrak{f}}
\renewcommand{\a}{\mathfrak{a}}
\renewcommand{\b}{\mathfrak{b}}
\renewcommand{\c}{\mathfrak{c}}
\newcommand{\q}{\mathfrak{q}}
\renewcommand{\r}{\mathfrak{r}}

\newcommand{\AN}{\mathfrak{N}}

\newcommand{\cR}{\mathcal{R}}

\newcommand{\ul}{\underline}

\DeclareMathOperator{\supp}{supp}
\DeclareMathOperator{\Vol}{Vol}

\title{On rings of integers generated by their units}
\author{Christopher Frei}
\email{frei@math.tugraz.at}
\address{Institut f\"ur Mathematik A\\
Technische Universit\"at Graz\\
Steyrergasse 30, A-8010 Graz\\
Austria}
\date{}
\thanks{The author is supported by the Austrian Science Foundation (FWF) project S9611-N23.}

\begin{document}

\begin{abstract}
We give an affirmative answer to the following question by Jarden and Narkiewicz: Is it true that every number field has a finite extension $L$ such that the ring of integers of $L$ is generated by its units (as a ring)?

As a part of the proof, we generalise a theorem by Hinz on power-free values of polynomials over number fields.
\end{abstract}

\maketitle

\section{Introduction}
The earliest result regarding the additive structure of units in rings of algebraic integers dates back to 1964, when Jacobson \cite{Jacobson1964} proved that every element of the rings of integers of $\Q(\sqrt{2})$ and $\Q(\sqrt{5})$ can be written as a sum of distinct units. Later, {\'S}liwa \cite{Sliwa1974} continued Jacobson's work, proving that there are no other quadratic number fields with that property, nor any pure cubic ones. Belcher \cite{Belcher1974}, \cite{Belcher1975} continued along these lines and investigated cubic and quartic number fields. 

In a particularly interesting lemma \cite[Lemma 1]{Belcher1974}, Belcher characterised all quadratic number fields whose ring of integers is generated by its units: These are exactly the fields $\Q(\sqrt{d})$, $d \in \Z$ squarefree, for which either
\begin{enumerate}
 \item $d \in \{-1, -3\}$, or
 \item $d > 0$, $d \not\equiv 1 \mod 4$, and $d+1$ or $d-1$ is a perfect square, or
 \item $d > 0$, $d \equiv 1 \mod 4$, and $d+4$ or $d-4$ is a perfect square.
\end{enumerate}
This result was independently proved again by Ashrafi and V{\'a}mos \cite{Ashrafi2005}, who also showed the following: Let $\O$ be the ring of integers of a quadratic or complex cubic number field, or of a cyclotomic number field of the form $\Q(\zeta_{2^n})$. Then there is no positive integer $N$ such that every element of $\O$ is a sum of $N$ units.

Jarden and Narkiewicz \cite{Jarden2007} proved a more general result which implies that the ring of integers of every number field has this property: If $R$ is a finitely generated integral domain of zero characteristic then there is no integer $N$ such that every element of $R$ is a sum of at most $N$ units. This also follows from a result obtained independently by Hajdu \cite{Hajdu2007}. The author \cite{Frei2010} proved an analogous version of this and of Belcher's result for rings of $S$-integers in function fields.

In \cite{Jarden2007}, Jarden and Narkiewicz raised three open problems:
\begin{enumerate}[A.]
 \item Give a criterion for an algebraic extension $K$ of the rationals to have the property that the ring of integers of $K$ is generated by its units.
 \item Is it true that each number field has a finite extension $L$ such that the ring of integers of $L$ is generated by its units?
 \item Let $K$ be an algebraic number field. Obtain an asymptotical formula for the number $N_k(x)$ of positive rational integers $n \leq x$ which are sums of at most $k$ units of the ring of integers of $K$.
\end{enumerate}

The result by Belcher stated above solves Problem A for quadratic number fields. Similar criteria have been found for certain types of cubic and quartic number fields \cite{Filipin2008}, \cite{Tichy2007}, \cite{Ziegler2008}. All these results have in common that the unit group of the ring in question is of rank $1$. 

Quantitative questions similar to Problem C were investigated in \cite{Filipin2008}, \cite{Filipin2008B}, \cite{Fuchs2009}. The property asked for in Problem B is known to hold for number fields with an Abelian Galois group, due to the Kronecker-Weber theorem. However, this is all that was known until recently, when the author \cite{Frei2010B} affirmatively answered the question in the function field case. In this paper, we use similar ideas to solve Problem B in its original number field version:

\begin{theorem}\label{extension}
For every number field $K$ there exists a number field $L$ containing $K$ such that the ring of integers of $L$ is generated by its units (as a ring).
\end{theorem}

It is crucial to our proof to establish the existence of integers of $K$ with certain properties (see Proposition \ref{sqfree}). We achieve this by asymptotically counting such elements. To this end, we need a generalised version of a theorem by Hinz \cite[Satz 1.1]{Hinz1982}, which is provided first. Let us start with some notation.
\section{Notation and auxiliary results}\label{section_notation}
All rings considered are commutative and with unity, and the ideal $\{0\}$ is never seen as a prime ideal. Two ideals $\a$, $\b$ of a ring $R$ are \emph{relatively prime} if $\a + \b = R$. Two elements $\alpha$, $\beta \in R$ are \emph{relatively prime} if the principal ideals $(\alpha)$, $(\beta)$ are.

The letter $K$ denotes a number field of degree $n > 1$, with discriminant $d_K$ and ring of integers $\O_K$. Let there be $r$ distinct real embeddings $\sigma_1$, $\ldots$, $\sigma_r : K \to \R$ and $2s$ distinct non-real embeddings $\sigma_{r+1}$, $\ldots$, $\sigma_{n} : K \to \C$, such that $\overline{\sigma_{r+j}} = \sigma_{r+s+j}$, for all $1 \leq j \leq s$. Then $\sigma : K \to \R^n$ is the standard embedding given by
\[\alpha \mapsto (\sigma_1(\alpha), \ldots, \sigma_r(\alpha), \Re\sigma_{r+1}(\alpha), \Im\sigma_{r+1}(\alpha), \ldots, \Re\sigma_{r+s}(\alpha), \Im\sigma_{r+s}(\alpha))\text.\]
An element $\alpha \in \O_K$ is called \emph{totally positive}, if $\sigma_i(\alpha) > 0$ for all $1 \leq i \leq r$.

A non-zero ideal of $\O_K$ is called \emph{$m$-free}, if it is not divisible by the $m$-th power of any prime ideal of $\O_K$, and an element $\alpha \in \O_K \setminus \{0\}$ is called \emph{$m$-free}, if the principal ideal $(\alpha)$ is $m$-free. We denote the \emph{absolute norm} of a non-zero ideal $\a$ of $\O_K$ by $\AN\a$, that is $\AN\a = [\O_K : \a]$. For non-zero ideals $\a$, $\b$ of $\O_K$, the ideal $(\a, \b)$ is their greatest common divisor. If $\beta \in \O_K\setminus\{0\}$ then we also write $(\a, \beta)$ instead of $(\a, (\beta))$. By $\supp \a$, we denote the set of all prime divisors of the ideal $\a$ of $\O_K$. The symbol $\mu$ stands for the Möbius function for ideals of $\O_K$. 

For $\underline{x} = (x_1, \ldots, x_n) \in \R^n$, with $x_i \geq 1$ for all $1 \leq i \leq n$, and $x_{r+s+i} = x_{r+i}$, for all $1 \leq i \leq s$, we define
\[\cR(\ul{x}) := \{\alpha \in \O_K \mid \alpha \text{ totally positive, }\left|\sigma_i(\alpha)\right| \leq x_i\text{ for all } 1\leq i \leq n\}\text,\]
and
\[x := x_1 \cdots x_n\text.\]

Let $f \in \O_K[X]$ be an irreducible polynomial of degree $g \geq 1$. For any ideal $\a$ of $\O_K$, let
\[L(\a) := \left|\{\beta + \a \in \O_K/\a \mid f(\beta) \equiv 0 \mod \a \}\right|\text.\]
By the Chinese remainder theorem, we have $L(\a_1\cdots \a_k) = L(\a_1)\cdots L(\a_k)$, for ideals $\a_1$, $\ldots$, $\a_k$ of $\O_K$ that are mutually relatively prime.

We say that the ideal $\a$ of $\O_K$ is a \emph{fixed divisor} of $f$ if $\a$ contains all $f(\alpha)$, for $\alpha \in \O_K$.

Hinz established the following result, asymptotically counting the set of all $\alpha \in \cR(\ul{x})$ such that $f(\alpha)$ is $m$-free:

\begin{theorem}[({\cite[Satz 1.1]{Hinz1982}})]\label{hinz}
If $m \geq \max\{2, \sqrt{2g^2+1}-(g+1)/2\}$, such that no $m$-th power of a prime ideal of $\O_K$ is a fixed divisor of $f$, then
\[\sum_{\substack{\alpha \in \cR(\ul{x})\\f(\alpha)\ m\text{-free}}}1 = \frac{(2\pi)^s}{\sqrt{\left|d_K\right|}}\cdot x\cdot \prod_{\P}\left(1 - \frac{L(\P^m)}{\AN\P^m} \right) + O(x^{1 - u})\text,\]
as $x$ tends to infinity. Here, $u = u(n, g)$ is an effective positive constant depending only on $n$ and $g$, the infinite product over all prime ideals $\P$ of $\O_K$ is convergent and positive, and the implicit $O$-constant depends on $K$, $m$ and $f$. 
\end{theorem}

A subring $\O$ of $\O_K$ is called an \emph{order} of $K$ if $\O$ is a free $\Z$-module of rank $[K : \Q]$, or, equivalently, $\Q\O = K$. Orders of $K$ are one-dimensional Noetherian domains. For any order $\O$ of $K$, the \emph{conductor} $\f$ of $\O$ is the largest ideal of $\O_K$ that is contained in $\O$, that is
\[\f = \{\alpha \in \O_K \mid \alpha \O_K \subseteq \O\}\text.\]
In particular, $\f \supsetneqq \{0\}$, since $\O_K$ is finitely generated as an $\O$-module. For more information about orders, see for example \cite[Section I.12]{Neukirch1999}.

Assume now that $f \in \O[X]$. Then we define, for any ideal $\a$ of $\O_K$,
\[L_\O(\a) := \left|\{\alpha + (\O \cap \a) \in \O/(\O\cap\a) \mid f(\alpha) \equiv 0 \mod (\O \cap \a)\}\right|\text.\]
The natural monomorphism $\O/(\O \cap \a) \to \O_K/\a$ yields $L_\O(\a) \leq L(\a)$, and if $\a_1$, $\ldots$, $\a_k$ are ideals of $\O_K$ such that all $\a_i \cap \O$ are mutually relatively prime then $L_\O(\a_1\cdots \a_k) = L_\O(\a_1)\cdots L_\O(\a_k)$.

In our generalised version of Theorem \ref{hinz}, we do not count all $\alpha \in \cR(\ul{x})$ such that $f(\alpha)$ is $m$-free, but all $\alpha \in \cR(\ul{x}) \cap \O$, such that $f(\alpha)$ is $m$-free and $\f(\alpha) \notin \P$, for finitely many given prime ideals $\P$ of $\O_K$.

\begin{theorem}\label{density}
Let $\O$ be an order of $K$ of conductor $\f$, and $f \in \O[X]$ an irreducible (over $\O_K$) polynomial of degree $g\geq 1$. Let $\mathcal{P}$ be a finite set of prime ideals of $\O_K$ that contains the set $\mathcal{P}_\f := \supp\f$. Let
\begin{equation}\label{mvsg}m \geq \max\left\{2,\sqrt{2g^2 +1}-(g+1)/2\right\}\end{equation}
be an integer such that no $m$-th power of a prime ideal of $\O_K$ is a fixed divisor of $f$, and denote by $N(\ul{x})$ the number of all $\alpha \in \O \cap \cR(\ul{x})$, such that
\begin{enumerate}
 \item for all $\P \in \mathcal{P}$, $f(\alpha) \notin \P$
 \item $f(\alpha)$ is $m$-free.
\end{enumerate}
Then 
\[N(\ul{x}) = D x + O(x^{1-u})\text,\]
as $x$ tends to infinity. Here, $u = u(n, g)$ is an explicitly computable positive constant that depends only on $n$ and $g$. The implicit $O$-constant depends on $K$, $\mathcal{P}$, $f$ and $m$. Moreover,
\[D = \frac{(2\pi)^s}{\sqrt{\left|d_K\right|}[\O_K : \O]}\sum_{\a \mid \f}\frac{\mu(\a)L_\O(\a)}{[\O : \a \cap \O]}\prod_{\P \in \mathcal{P}\setminus\mathcal{P_\f}}\left(1 - \frac{L(\P)}{\AN\P}\right)\prod_{\P \notin \mathcal{P}}\left(1 - \frac{L(\P^m)}{\AN\P^m} \right)\text.\]
The sum runs over all ideals of $\O_K$ dividing $\f$, and the infinite product over all prime ideals $\P \notin\mathcal{P}$ of $\O_K$ is convergent and positive.
\end{theorem}

For our application, the proof of Theorem \ref{extension}, we only need the special case where $m=g=2$, and we do not need any information about the remainder term. However, the additional effort is small enough to justify a full generalisation of Theorem \ref{hinz}, instead of just proving the special case. The following proposition contains all that we need of Theorem \ref{density} to prove Theorem \ref{extension}.

\begin{proposition}\label{sqfree}
Assume that for every prime ideal of $\O_K$ dividing $2$ or $3$, the relative degree is greater than $1$, and that $\O \neq \O_K$ is an order of $K$. Let $\mathcal{P}$ be a finite set of prime ideals of $\O_K$, and let $\eta \in \O \smallsetminus K^2$. Then there is an element $\omega \in \O_K$ with the following properties:
\begin{enumerate}
 \item $\omega \notin \O$,
 \item for all $\P \in \mathcal{P}$, $\omega^2 - 4\eta \notin \P$, and
 \item $\omega^2 - 4\eta$ is squarefree.
\end{enumerate}
\end{proposition}

The basic idea to prove Theorem \ref{extension} is as follows: Let $\O$ be the ring generated by the units of $\O_K$. With Proposition \ref{sqfree}, we find certain elements $\omega_1$, $\ldots$, $\omega_r$ of $\O_K$, such that $\O[\omega_1, \ldots, \omega_r] = \O_K$. Due to the special properties from Proposition \ref{sqfree}, we can construct an extension field $L$ of $K$, such that $\omega_1$, $\ldots$, $\omega_r$ are sums of units of $\O_L$, and $\O_L$ is generated by units as a ring extension of $\O_K$. This is enough to prove that $\O_L$ is generated by its units as a ring.

\section{Proof of Theorem 3}
We follow the same strategy as Hinz \cite{Hinz1982} in his proof of Theorem \ref{hinz}, with modifications where necessary. For any vector $v \in \R^n$, we denote its Euclidean length by $\left|v\right|$. We use a theorem by Widmer to count lattice points:

\begin{theorem}[({\cite[Theorem 5.4]{Widmer2010}})]\label{counting_lattice_points}
Let $\Lambda$ be a lattice in $\R^n$ with successive minima (with respect to the unit ball) $\lambda_1$, $\ldots$, $\lambda_n$. Let $B$ be a bounded set in $\R^n$ with boundary $\partial B$. Assume that there are $M$ maps $\Phi : [0, 1]^{n-1} \to \R^n$ satisfying a Lipschitz condition
\[\left|\Phi(v) - \Phi(w)\right| \leq L\left|v - w\right|\text,\]
such that $\partial B$ is covered by the union of the images of the maps $\Phi$. Then $B$ is measurable, and moreover
\[\left|\left|B \cap \Lambda\right| - \frac{\Vol B}{\det \Lambda} \right| \leq c_0(n) M \max_{0\leq i<n}\frac{L^i}{\lambda_1\cdots \lambda_i}\text.\]
For $i=0$, the expression in the maximum is to be understood as $1$. Furthermore, one can choose $c_0(n) = n^{3n^2/2}$.
\end{theorem}

We need some basic facts about contracted ideals in orders. The statements of the following lemma can hardly be new, but since the author did not find a reference we shall prove them for the sake of completeness.

\begin{lemma}\label{orders}
Let $\O \subseteq \O_K$ be an order of $K$ with conductor $\f$. Then, for any ideals $\a$, $\b$ of $\O_K$, the following holds:
\begin{enumerate}[(1)]
 \item if $\a + \f = \O_K$ and $\b \mid \f$ then $(\a \cap \O) + (\b \cap \O) = \O$.
 \item if $\a + \f = \O_K$, $\b + \f = \O_K$, and $\a + \b = \O_K$ then $(\a \cap \O) + (\b \cap \O) = \O$.
 \item if $\a + \f = \O_K$ then $[\O : \a \cap \O] = \AN\a$.
\end{enumerate}
\end{lemma}

\begin{proof}
For any ideal $\a$ of $\O_K$ with $\a + \f = \O_K$, we have
\[(\a \cap \O) + \f = (\a + \f) \cap \O = \O_K \cap \O = \O\text.\]
The first equality holds because for every $\alpha \in \a$, $\beta \in \f \subseteq \O$ with $\alpha + \beta \in \O$ it follows that $\alpha \in \O$.

Moreover, if $\c$ is an ideal of $\O$ with $\c + \f = \O$ then
\[\c \O_K + \f \supseteq (\c + \f)\O_K = \O \O_K = \O_K\text.\]

Therefore,
\[\varphi : \a \mapsto \a \cap \O\text{ and }\psi : \c \mapsto \c\O_K\]
are maps between the sets of ideals 
\[\{\a \subseteq \O_K \mid \a + \f = \O_K\} \text{ and } \{\c \subseteq \O \mid \c +\f = \O\}\text.\]

Let us prove that $\varphi$ and $\psi$ are inverse to each other. Clearly, $(\varphi \circ \psi)(\c) \supseteq \c$ and $(\psi \circ \varphi)(\a) \subseteq \a$. Also,
\[(\varphi\circ\psi)(\c) = (\c\O_K \cap \O)\O = (\c\O_K\cap\O)(\c+\f)\subseteq\c + \f(\c\O_K \cap \O) \subseteq \c+\c\f\O_K \subseteq \c\text,\]
and
\[\a = \a\O = \a((\a\cap\O) + \f) \subseteq (\a\cap\O)\O_K + \f\a \subseteq(\a\cap\O)\O_K + (\a \cap \O) = (\psi\circ\varphi)(\a)\text.\]

Clearly, $\varphi$ and $\psi$ are multiplicative, so the monoid of ideals of $\O$ relatively prime to $\f$ is isomorphic with the monoid of ideals of $\O_K$ relatively prime to $\f$. (In the special case where $\O$ is an order in an imaginary quadratic field this is proved in \cite[Proposition 7.20]{Cox1989}.)

If $\a$, $\b$ are as in \emph{(1)} then $\f \subseteq \b \cap \O$, and thus $\O = (\a \cap \O) + \f \subseteq (\a \cap \O) + (\b \cap \O)$.

Suppose now that $\a$, $\b$ are as in \emph{(2)}, and $\varphi(\a) + \varphi(\b) =: \c \subseteq \O$. Then $\c + \f \supseteq \varphi(\a) + \f = \O$, whence $\c = \varphi(\mathfrak{d})$, for some ideal $\mathfrak{d}$ of $\O_K$ relatively prime to $\f$. Now $\a \subseteq \mathfrak{d}$ and $\b \subseteq \mathfrak{d}$, so $\mathfrak{\d} = \O_K$, and thus $\c = \O$.

To prove \emph{(3)}, we show that the natural monomorphism $\Phi : \O/(\a \cap \O) \to \O_K/\a$ is surjective. This holds true, since
\[\O_K = \a + \f \subseteq \a + \O\text.\]
\end{proof}

For now, let us prove Theorem \ref{density} with the additional assumption that $f(\alpha) \neq 0$ for all totally positive $\alpha \in \O_K$. This holds of course if $\deg f \geq 2$, since $f$ is irreducible over $\O_K$. At the end of the proof, we specify the changes necessary to drop this assumption. Let
\[\Pi := \prod_{\P \in \mathcal{P}}\P\text.\]
It is well known that 
\[\sum_{\a \mid \b}\mu(\a) = \begin{cases}1, &\text{ if } \b = \O_K\\0, &\text{ otherwise,} \end{cases}\]
for any nonzero ideal $\b$ of $\O_K$. Assume that $f(\alpha) \neq 0$. Then 
\[\sum_{\a \mid (\Pi, f(\alpha))}\mu(\a) = \begin{cases}1, &\text{ if for all } \P\in \mathcal{P},\ f(\alpha) \notin \P \\0, &\text{ otherwise.} \end{cases}\]
Write $(f(\alpha)) = \c_1 \c_2^m$, where $\c_1$ is $m$-free. Then $\b^m \mid f(\alpha)$ if and only if $\b \mid \c_2$, whence
\[\sum_{\b^m \mid f(\alpha)}\mu(\b) = \begin{cases}1, &\text{ if $f(\alpha)$ is $m$-free}\\0, &\text{ otherwise.} \end{cases}\]
Therefore,
\begin{equation}\label{n}N(\ul{x}) = \sum_{\alpha \in \cR(\ul{x}) \cap \O}\sum_{\a \mid (\Pi, f(\alpha))}\mu(\a)\sum_{\b^m \mid f(\alpha)}\mu(\b)\text.\end{equation}
Put 
\begin{equation}\label{n1}N_1(\ul{x}, y) := \sum_{\alpha \in \cR(\ul{x}) \cap \O}\sum_{\a \mid (\Pi, f(\alpha))}\mu(\a)\sum_{\substack{(\b,\Pi)=1\\\b^m \mid f(\alpha)\\\AN\b \leq y}}\mu(\b)\text,\end{equation}
and
\begin{equation}\label{n2}N_2(\ul{x}, y) := \sum_{\alpha \in \cR(\ul{x}) \cap \O}\sum_{\a \mid (\Pi, f(\alpha))}\mu(\a)\sum_{\substack{\b^m \mid f(\alpha)\\\AN\b > y}}\mu(\b)\text.\end{equation}
It will turn out that, with a suitable choice of $y$, the main component of $N(\ul{x})$ is $N_1(\ul{x}, y)$. In fact, since
\[\sum_{\a\mid(\Pi, f(\alpha))}\mu(\a)\sum_{\substack{(\b, \Pi) \neq 1\\\b^m \mid f(\alpha)\\\AN\b \leq y}}\mu(\b) = 0\text,\]
for all $\alpha \in \O_K$ with $f(\alpha)\neq 0$, we have
\begin{equation}\label{splitn}N(\ul{x}) = N_1(\ul{x}, y) + N_2(\ul{x}, y)\text.\end{equation}

\subsection{Estimation of $N_2(\ul{x}, y)$}\label{section_n2est}

We can reduce the estimation of $N_2(\ul{x}, y)$ to a similar computation to that which has already been performed by Hinz \cite{Hinz1982}. Indeed, for any nonzero ideal $\q$ of $\O_K$, we have
\begin{align*}
\left|N_2(\ul{x}, y)\right| &\leq \sum_{\alpha \in \cR(\ul{x}) \cap \O}\big|\sum_{\a \mid (\Pi, f(\alpha))}\mu(\a)\big|\cdot\big|\sum_{\substack{\b^m \mid f(\alpha)\\\AN\b>y}}\mu(\b)\big|\\
&\leq\big(\sum_{\a \mid \Pi}\mu(\a)^2\big)\sum_{\alpha \in \cR(\ul{x})}\big|\sum_{\c \mid \q}\sum_{\substack{\b^m\mid f(\alpha)\\\AN\b>y\\(\b, \q)=\c}}\mu(\b)\big|\\
&\leq\AN\Pi\AN\q\sum_{\alpha \in \cR(\ul{x})}\sum_{\substack{\b^m\mid f(\alpha)\\\AN\b>y/\AN\q\\(\b, \q)=1}}\mu(\b)^2\text.
\end{align*}
The last expression differs only by a multiplicative constant from the right-hand side of \cite[(2.6)]{Hinz1982}, so we can use Hinz's estimates \cite[pp. 139-145]{Hinz1982} without any change. With a suitable choice of $\q$ (\cite[(2.8)]{Hinz1982}), we get (see Lemma 2.2 and the proof of Theorem 2.1 from \cite{Hinz1982})
\begin{equation}\label{n2est}N_2(\ul{x}, y) = O(x^{g/(2l+1)} y^{(l-m)/(2l+1)} (x y^{(l-m)/g} + 1))\text,\end{equation}
for any integer $1 \leq l \leq m-1$, as $x$, $y \to \infty$. The implicit $O$-constant depends on $K$, $f$, $m$, and $\mathcal{P}$.

\subsection{Computation of $N_1(\ul{x}, y)$}

Now let us compute $N_1(\ul{x}, y)$. We have
\begin{equation}\label{n1comp1}N_1(\ul{x}, y) = \sum_{\a \mid \Pi}\mu(\a)\sum_{\substack{(\b, \Pi)=1\\\AN\b\leq y}}\mu(\b)\left|M_{\a, \b}(\ul{x}) \right|\text,\end{equation}
where $M_{\a, \b}(\ul{x})$ is the set of all $\alpha \in \cR(\ul{x})\cap\O$ such that $f(\alpha)\in \a$ and $f(\alpha)\in \b^m$. Since all occurring ideals $\a$, $\b$ are relatively prime, we have
\begin{align*}
M_{\a, \b}(\ul{x}) &= \{\alpha \in \cR(\ul{x}) \cap \O \mid f(\alpha) \equiv 0 \mod \a\b^m\} \\&= \bigcup_{\substack{\beta + \a\b^m \in \O_K/\a\b^m\\f(\beta) \equiv 0 \mod \a\b^m}}\left((\beta + \a\b^m) \cap \cR(\ul{x}) \cap \O \right)\text,
\end{align*}
where the union over all roots of $f$ modulo $\a\b^m$ is disjoint. We asymptotically count each of the sets $(\beta + \a\b^m) \cap \cR(\ul{x}) \cap \O$ by counting lattice points. Consider the natural monomorphism $\varphi : \O/(\a\b^m \cap \O) \to \O_K/\a\b^m$, mapping $\alpha + (\a\b^m\cap\O)$ to $\alpha + \a\b^m$. 

\begin{lemma}\label{counting_integers}
The set $(\beta + \a\b^m) \cap \O$ is not empty if and only if $\beta + \a\b^m$ is in the image of $\varphi$.

In that case, let $\varepsilon \in [0, 1/n]$, and $c \geq 1/m$ such that $\AN\b \leq x^c$. Then
\[\left|\left|(\beta + \a\b^m)\cap\cR(\ul{x})\cap\O\right| - c_1(K)\frac{x}{[\O_K : \a\b^m \cap \O]}\right| \leq c_2(K)\frac{x^{1 -\varepsilon}}{\AN\b^{(1 - \varepsilon)/c}}\text.\]
Here, $c_1(K) = (2\pi)^s/\sqrt{\left|d_K\right|}$, and $c_2(K)$ is an explicitly computable constant which depends only on $K$.
\end{lemma}

\begin{proof}
If $\alpha \in (\beta + \a\b^m) \cap \O$ then $\beta + \a\b^m = \alpha + \a\b^m = \varphi(\alpha + (\a\b^m\cap \O))$. If, on the other hand, $\beta + \a\b^m = \varphi(\alpha + (\a\b^m\cap\O))$, for some $\alpha \in \O$, then $\alpha + \a\b^m = \beta + \a\b^m$, and thus $\alpha \in (\beta + \a\b^m) \cap \O$.

Assume now that $(\beta + \a\b^m) \cap \O$ is not empty. Then, for any $\alpha \in (\beta + \a\b^m) \cap \O$, we have
\[\left|(\beta + \a\b^m)\cap\cR(\ul{x})\cap\O\right| = \left|(\a\b^m \cap \O) \cap (\cR(\ul{x}) - \alpha)\right|\text.\]
Let $\sigma : K \to \R^n$ be the standard embedding defined in Section \ref{section_notation}, and let $T : \R^n \to \R^n$ be the linear automorphism given by
\begin{align*}
T(e_i) &= x^{1/n}/x_i \cdot e_i\text{, for $1 \leq i \leq r$, and }\\
T(e_{r+i}) &= x^{1/n}/x_{r+\lceil i/2\rceil}\cdot e_{r + i}\text{, for $1 \leq i \leq 2s$}\text,
\end{align*}
where $e_1$, $\ldots$, $e_n$ is the standard basis of $\R^n$. Then
\begin{equation}\label{dett}\det T = x/(x_1\cdots x_r x_{r+1}^2 \cdots x_{r+s}^2) = x/(x_1 \cdots x_n) = 1\text.\end{equation}
Therefore,  $T(\sigma(\a\b^m \cap \O))$ is a lattice in $\R^n$ with determinant 
\begin{equation}\label{det}\det T(\sigma(\a\b^m \cap \O)) = 2^{-s}\sqrt{\left|d_K\right|}[\O_K : \a\b^m \cap \O]\text.\end{equation}
Moreover, $T(\sigma(\cR(\ul{x}) - \alpha)) = T(\sigma(\O_K)) \cap B$, where $B$ is a product of $r$ line segments of length $x^{1/n}$ and $s$ disks of radius $x^{1/n}$. Clearly,
\begin{equation}\label{vol}\Vol(B) = \pi^s x\text.\end{equation} 
We construct maps $\Phi : [0,1]^{n-1} \to \R^n$ as in Theorem \ref{counting_lattice_points}. Write $B = l_1 \times \cdots \times l_r \times d_{r+1} \times \cdots \times d_{r+s}$, with line segments $l_i$ of length $x^{1/n}$ and disks $d_i$ of radius $x^{1/n}$. Put 
\[B_i := l_1 \times \cdots \times l_{i-1} \times (\partial l_i) \times l_{i+1} \times \cdots \times l_r \times d_{r+1} \times \cdots \times d_{r+s}\text,\]
for $1 \leq i \leq r$, and
\[B_i := l_1 \times \cdots \times l_r \times d_{r+1} \times \cdots \times d_{i-1} \times (\partial d_i) \times d_{i+1} \times \cdots \times d_{r+s}\text,\]
for $r+1 \leq i \leq r+s$. Then
\[\partial B = \bigcup_{i=1}^{r+s} B_i\text.\]
For $1 \leq i \leq r$, $\partial l_i$ consists of two points, and the remaining factor of $B_i$ is contained in an $(n-1)$-dimensional cube of edge-length $2x^{1/n}$. For $r+1 \leq i \leq r+s$, $\partial d_i$ is a circle of radius $x^{1/n}$, and the remaining factor of  $B_i$ is contained in an $(n-2)$-dimensional cube of edge-length $2x^{1/n}$. Therefore, we find $2r + s$ maps $\Phi : [0, 1]^{n-1} \to \R^n$ with
\begin{equation}\label{lipschitz}\left|\Phi(v) - \Phi(w)\right| \leq 2 \pi x^{1/n} \left|v - w\right|\text,\end{equation}
such that $\partial B$ is covered by the union of the images of the maps $\Phi$.

Since
\begin{align*}
\left|(\beta + \a\b^m)\cap\cR(\ul{x})\cap\O\right| &= \left|T(\sigma(\a\b^m \cap \O)) \cap T(\sigma(\cR(\ul{x}) - \alpha))\right|\\&= \left|T(\sigma(\a\b^m \cap \O)) \cap B\right|\text,
\end{align*}
Theorem \ref{counting_lattice_points} and \eqref{det}, \eqref{vol}, \eqref{lipschitz} yield
\begin{equation}\label{lattice_point_estimate}\left|\left|(\beta + \a\b^m)\cap\cR(\ul{x})\cap\O\right| - \frac{(2\pi)^s}{\sqrt{\left|d_K\right|}} \frac{x}{[\O_K : \a\b^m \cap \O]}\right| \leq c_3(K)\frac{x^{i/n}}{\lambda_1 \cdots \lambda_i}\text.\end{equation}
Here, $c_3(K) = (2r+s)(2\pi)^{n-1}n^{3n^2/2}$, $i \in \{0, \ldots, n-1\}$, and $\lambda_1$, $\ldots$, $\lambda_i$ are the first $i$ successive minima of the lattice $T(\sigma(\a\b^m\cap\O))$ with respect to the unit ball.

Let us further estimate the right-hand side of \eqref{lattice_point_estimate}. First, we need a lower bound for $\lambda_i$ in terms of $\AN\b$. For each $i$, there is some $\alpha \in (\a\b^m \cap \O) \smallsetminus\{0\}$ with $\lambda_i = \left|T(\sigma(\alpha))\right|$. Since $\alpha \in \b^m$, the inequality of weighted arithmetic and geometric means and \eqref{dett} yield (cf. \cite[Lemma 5]{Masser2006}, \cite[Lemma 9.7]{Widmer2010})
\begin{align*}
\AN\b^m &\leq \left|N(\alpha)\right| = \prod_{j=1}^n \left|\sigma_j(\alpha)\right| = \prod_{j=1}^{r+s} \left|\frac{x^{1/n}}{x_j} \sigma_j(\alpha)\right|^{d_j}\\&\leq \left(\frac{1}{n} \sum_{j=1}^{r+s}d_j\left|\frac{x^{1/n}}{x_j}\sigma_j(\alpha)\right|^2\right)^{n/2}\leq \left(\frac{2}{n}\right)^{n/2}\lambda_i^n\text.
\end{align*}
Here, $d_j = 1$ for $1\leq j \leq r$, and $d_j = 2$ for $r+1 \leq j \leq r+s$. Recall that $n \geq 2$. With the assumptions on $\varepsilon$ and $c$ in mind, we get
\begin{align*}\frac{x^{i/n}}{\lambda_1 \cdots \lambda_i} &\leq \left(\frac{2}{n}\right)^{i/2} \frac{x^{i/n}}{\AN\b^{mi/n}} \leq  \frac{x^{1-\varepsilon}}{\AN\b^{mi/n + (1 - \varepsilon - i/n)/c}} \leq \frac{x^{1-\varepsilon}}{\AN\b^{(1 - \varepsilon)/c}}\text.
\end{align*}
\end{proof}

Since $f \in \O[X]$, we can conclude from $\beta + \a\b^m = \varphi(\alpha + (\a\b^m \cap \O))$ that $f(\beta)\in\a\b^m$ if and only if $f(\alpha) \in \a\b^m \cap \O$. Therefore,
\[M_{\a, \b}(\ul{x}) = \bigcup_{\substack{\alpha + (\a\b^m\cap\O) \in \O/(\a\b^m\cap\O)\\f(\alpha) \equiv 0 \mod (\a\b^m\cap\O)}}\left((\alpha + \a\b^m) \cap \O \cap \cR(\ul{x})\right)\text,\]
and thus
\[\left|\left|M_{\a,\b}(\ul{x})\right| - c_1(K) L_\O(\a\b^m) \frac{x}{[\O_K : \a\b^m \cap \O]}\right| \leq c_2(K) L(\a)L(\b^m)\frac{x^{1-\varepsilon}}{\AN\b^{(1 - \varepsilon)/c}}\text,\]
whenever $\AN\b \leq x^c$, for some $c \geq 1/m$, and $\varepsilon \in [0, 1/n]$. Notice that $L_\O(\a\b^m) \leq L(\a\b^m) = L(\a)L(\b^m)$, since $\a$, $\b$ are relatively prime. Therefore,
\begin{align*}
&\big|\sum_{\substack{(\b, \Pi) = 1\\\AN\b\leq x^{c}}}\mu(\b)\left|M_{\a,\b}(\ul{x})\right| - c_1(K)x\sum_{(\b,\Pi)=1}\mu(\b)\frac{L_\O(\a\b^m)}{[\O_K : \a\b^m \cap \O]}\big|\\
\leq\ &\big|\sum_{\substack{(\b, \Pi) = 1\\\AN\b\leq x^{c}}}\mu(\b)\left(\left|M_{\a,\b}(\ul{x})\right|-c_1(K) x \frac{L_\O(\a\b^m)}{[\O_K : \a\b^m \cap \O]}\right)\big| \\
+\ &\big|c_1(K) x \sum_{\substack{(\b, \Pi) = 1\\\AN\b > x^{c}}}\mu(\b)\frac{L_\O(\a\b^m)}{[\O_K : \a\b^m \cap \O]}\big|\\
\leq\ &c_2(K) x^{1 - \varepsilon} L(\a) \sum_{(\b,\Pi)=1}\mu(\b)^2\frac{L(\b^m)}{\AN\b^{(1 - \varepsilon)/c}}\\
+\ &c_1(K) L(\a) x \sum_{\substack{(\b,\Pi)=1\\\AN\b>x^{c}}}\mu(\b)^2 \frac{L(\b^m)}{[\O_K : \a\b^m \cap \O]}\text.
\end{align*}
Let $s>1$ be a real number. As in \cite[top of p. 138]{Hinz1982}, we get
\[\sum_{\AN\b \leq y}\mu(\b)^2 L(\b^m) = O(y)\text,\]
whence
\[\sum_{\substack{(\b,\Pi)=1\\\AN\b>x^{c}}}\mu(\b)^2 \frac{L(\b^m)}{\AN\b^s} = O(x^{c(1-s)})\text,\]
by partial summation. Therefore, the sum
\[\sum_{(\b,\Pi)=1}\mu(\b)^2\frac{L(\b^m)}{\AN\b^{(1 - \varepsilon)/c}}\]
converges whenever $c < 1 - \varepsilon$. Since $[\O_K : \a\b^m \cap \O] \geq \AN\b^m$, we have
\[\sum_{\substack{(\b,\Pi)=1\\\AN\b>x^{c}}}\mu(\b)^2 \frac{L(\b^m)}{[\O_K : \a\b^m \cap \O]} \leq \sum_{\substack{(\b,\Pi)=1\\\AN\b>x^{c}}}\mu(\b)^2 \frac{L(\b^m)}{\AN\b^m} = O(x^{c(1-m)})\text.\]
Putting everything together, we get 
\begin{equation}\label{n1comp2}\begin{split}
\sum_{\substack{(\b, \Pi) = 1\\\AN\b\leq x^{c}}}\mu(\b)\left|M_{\a,\b}(\ul{x})\right| &= c_1(K)x\sum_{(\b,\Pi)=1}\mu(\b)\frac{L_\O(\a\b^m)}{[\O_K : \a\b^m \cap \O]}\\ &+ O(x^{1-\varepsilon} + x^{1 + c(1-m)})\text,\end{split}
\end{equation}
whenever $1/m \leq c < 1 - \varepsilon$ and $0 \leq \varepsilon \leq 1/n$, as $x \to \infty$. The implicit $O$-constant depends on $K$, $\a$, $\mathcal{P}$, $f$, $m$, $c$ and $\varepsilon$.

\subsection{End of the proof}\label{end}

By \eqref{splitn}, \eqref{n2est}, \eqref{n1comp1} and \eqref{n1comp2}, we get
\begin{align*}
N(\ul{x}) &= N_1(\ul{x}, x^{c}) + N_2(\ul{x}, x^{c})\\
     &= c_1(K)x\sum_{a\mid\Pi}\mu(\a)\sum_{(\b,\Pi)=1}\mu(\b)\frac{L_\O(\a\b^m)}{[\O_K : \a\b^m \cap \O]} + R\\ 
&=: D x + R\text,
\end{align*}
where 
\[R = O(x^{1-\varepsilon} + x^{1 -c(m-1)} + x^{g/(2l+1)-c(m-l)/(2l+1)} (x^{1-c(m-l)/g} + 1))\]
holds for every $0 \leq \varepsilon \leq 1/n$,  $1/m \leq c < 1-\varepsilon$, and $l \in \{1, \ldots, m-1\}$, as $x \to \infty$. The implicit $O$-constant depends on $K$, $\mathcal{P}$, $f$, $m$, $c$, and $\varepsilon$.

Assume first that $m > g+1$. Then we put 
\[l := m - g,\quad c := 1-5/(g+10),\quad \varepsilon := \min\{1/n,\ 4/(g+10)\}\text,\]
to get
\[R = O(x^{1-1/n} + x^{1-4/(g+10)} + x^{1 - g(g + 5)/(g+10)} + x^{(g+5)/(g+10)}) = O(x^{1-u(n,g)})\text,\]
with $u(n, g)$ as in the theorem.

Now suppose that $2 \leq m \leq g+1$. Then 
\[R = O(x^{1-\varepsilon} + x^{1 -c(m-1)} + x^{1 + g/(2l+1) - c(m - l)(g + 2l + 1)/(g(2l + 1))})\text.\]
We proceed as in \cite[Section 3, Proof of Theorem 1.1]{Hinz1982}. For every $m$ that satisfies \eqref{mvsg}, we find some $1 \leq l \leq m-1 \leq g$, such that $m - l > g^2/(2l+g+1)$. Then we can choose some $c$, depending only on $g$, $l$, with
\[\frac{1}{m} \leq \frac{g(2l+2)}{g(2l+2)(m-l+1)} \leq \frac{g(2l+1)+g^2}{(m-l)(2l+g+1)+g(2l+1)} \leq c < 1\text.\]
A straightforward computation shows that
\[1 + g/(2l+1) - c(m - l)(g + 2l + 1)(g(2l + 1)) \leq c\text.\]
For any $0 < \varepsilon < 1-c$, $\varepsilon \leq 1/n$, we get
\[R = O(x^{1-\varepsilon} + x^{1-c} + x^{c}) = O(x^{1 - u(n, g)})\text,\]
for a suitable choice of $u(n ,g)$. Notice that there are only finitely many values of $m$ for every $g$.

The only task left is to prove that $D$ has the form claimed in the theorem. We split up $D$ in the following way: Let $\Pi_1$ be the product of all prime ideals in $\mathcal{P}\setminus \mathcal{P}_\f$. Then
\begin{align*}
D &= c_1(K)\sum_{\a \mid \f}\mu(\a)\sum_{\b \mid \Pi_1}\mu(\b)\sum_{(\c,\Pi) = 1}\frac{\mu(\c)L_\O(\a\b\c^m)}{[\O_K : \a\b\c^m \cap \O]}\\
&= \frac{c_1(K)}{[\O_K : \O]}\sum_{\a \mid \f}\frac{\mu(\a) L_\O(\a)}{[\O : \a \cap \O]}\sum_{\b \mid \Pi_1}\frac{\mu(\b) L_\O(\b)}{[\O : \b \cap \O]}\sum_{(\c,\Pi) = 1}\frac{\mu(\c) L_\O(\c^m)}{[\O : \c^m \cap \O]}\text.
\end{align*}
This holds because for all combinations of $\a$, $\b$, $\c$ as above, the $\O$-ideals $(\a \cap \O)$, $(\b \cap \O)$ and $(\c^m \cap \O)$ are relatively prime to each other, by Lemma~\ref{orders}. Therefore,
\[[\O_K : \a\b\c^m \cap \O] = [\O_K : \O][\O : \a \cap \O][\O : \b \cap \O][\O : \c^m \cap \O]\text,\]
and
\[L_\O(\a\b\c^m) = L_\O(\a)L_\O(\b)L_\O(\c^m)\text.\]
Finally, we notice that, by Lemma \ref{orders}, $[\O : \r \cap \O] = \AN\r$ and thus  $L_\O(\r) = L(\r)$, for any ideal $\r$ of $\O_K$ relatively prime to $\f$. A simple Euler product expansion yields the desired form of $D$. All factors of the infinite product
\[\prod_{\P \notin \mathcal{P}}\left(1-\frac{L(\P^m)}{\AN\P^m}\right)\]
are positive, since no $\P^m$ is a fixed divisor of $f$. For all but the finitely many prime ideals of $\O_K$ that divide the discriminant of $f$, we have $L(\P^m) = L(\P) \leq g$. Therefore, the infinite product is convergent and positive.

This concludes the proof of Theorem \ref{density} under the assumption that $f$ has no totally positive root in $K$. If $f$ has such a root then we let the first sum in \eqref{n}, \eqref{n1}, \eqref{n2} run over all $\alpha \in \cR(\ul{x})\cap \O$ such that $f(\alpha) \neq 0$. The estimation of $N_2(\ul{x}, y)$ in Section \ref{section_n2est} holds still true, since a possible $\alpha$ with $f(\alpha) = 0$ is ignored in Hinz's estimates anyway. In \eqref{n1comp1}, we get an error term $O(y)$. This additional error term becomes irrelevant in Section \ref{end}. 

\section{Proof of Proposition 4}
We need the following estimate for the index $[\O_K : \O]$.

\begin{lemma}\label{index_of_order}
Let $\p_1$, $\ldots$, $\p_k$ be distinct prime ideals of $\O$. For each $1 \leq i \leq k$, let
\[\p_i \O_K = \P_{i, 1}^{e_{i,1}}\cdots\P_{i, l_i}^{e_{i, l_i}}\]
be the factorisation of $\p_i$ in $\O_K$, with distinct prime ideals $\P_{i,j}$ of $\O_K$, and $e_{i,j}$, $l_i\geq 1$.
Then
\[[\O_K : \O] \geq \prod_{i=1}^k\frac{1}{[\O : \p_i]}\prod_{j=1}^{l_i}\AN\P_{i,j}^{e_{i,j}}\text,\]
with equality if and only if $\f$ divides $\prod_{i = 1}^k\prod_{j=1}^{l_i}\P_{i,j}^{e_{i,j}}$.
\end{lemma}

\begin{proof}
Put
\[\Pi := \prod_{i=1}^k\prod_{j=1}^{l_i}\P_{i, j}^{e_{i,j}}\text.\]
Then we have
\[[\O_K : \O] = \frac{[\O_K : \Pi][\Pi : \Pi \cap \O]}{[\O : \Pi \cap \O]} \geq \frac{\AN\Pi}{[\O : \bigcap_{i=1}^k\p_i]}\\ = \frac{\prod_{i=1}^k\prod_{j=1}^{l_i}\AN\P_{i,j}^{e_{i,j}}}{\prod_{i=1}^k[\O : \p_i]}\text,\]
since $[\O : \Pi \cap \O] = [\O : \bigcap_{i=1}^k \p_i] = \prod_{i=1}^k[\O : \p_i]$, by the Chinese remainder theorem. Moreover, we have $\Pi = \Pi \cap \O$ if and only if $\f$ divides $\Pi$. 
\end{proof}

Without loss of generality, we may assume that $\mathcal{P}$ contains all prime ideals of $\O_K$ dividing the conductor $\f$ of $\O$. Since $\eta \in \O \smallsetminus K^2$, the polynomial $f := X^2 - 4 \eta \in \O[X]$ is irreducible over $\O_K$. Evaluating $f$ at $0$ and $1$, we see that the only fixed divisor of $f$ is $(1)$.

We put $x_1 = \cdots = x_n$, so
\[\cR(\ul{x}) = \{\alpha \in \O_K \mid \alpha \text{ totally positive, } \max_{1\leq i \leq n}\left|\sigma_i(\alpha)\right| \leq x^{1/n}\}\]
depends only on $x$. Let $N(x)$ be the number of all $\alpha \in \cR(\ul{x})$, such that
\begin{enumerate}
 \item for all $\P \in \mathcal{P}$, $\alpha^2 - 4\eta \notin \P$, and
 \item $\alpha^2 - 4\eta$ is squarefree,
\end{enumerate}
and let $N_\O(x)$ be the number of all $\alpha \in \cR(\ul{x}) \cap \O$ with the same two properties. 

Theorem \ref{density}, with $m=g=2$, invoked once with the maximal order $\O_K$ and once with the order $\O$, yields
\[N(x) = D x + O(x^{1-u})\ \text{ and }\ N_\O(x) = D_\O x + O(x^{1-u})\text.\]

To prove the proposition, it is enough to show that
\[\lim_{x\to\infty}\frac{N_\O(x)}{x} < \lim_{x\to\infty}\frac{N(x)}{x}\text,\]
that is, $D_\O < D$.

By Theorem \ref{density}, the infinite product 
\[\prod_{\P \notin \mathcal{P}}\left(1- \frac{L(\P^2)}{\AN\P^2}\right)\]
is convergent and positive. Moreover, we notice that 
\begin{equation}\label{greateronehalf}\left(1 - L(\P)/\AN\P\right) > 1/2\text,\end{equation}
for every prime ideal $\P$ of $\O_K$. This is obvious if $2 \notin \P$, since then $\AN\P\geq 5$ by the hypotheses of the proposition, but $f$ is of degree $2$, so $L(\P) \leq 2$. If $2 \in \P$ then we have $f \equiv X^2 \mod \P$, whence $L(\P) = 1$. On the other hand, $\AN\P \geq 4$, so \eqref{greateronehalf} holds again. Therefore, the finite product 
\[\prod_{\P \in \mathcal{P}\setminus\mathcal{P_\f}}\left(1 - \frac{L(\P)}{\AN\P} \right)\]
is positive as well. The proposition is proved if we can show that
\begin{equation}\label{corollary_goal}\frac{1}{[\O_K : \O]}\sum_{\a \mid \f}\frac{\mu(\a)L_\O(\a)}{[\O : \a \cap \O]} < \prod_{\P \in \mathcal{P}_\f}\left(1 - \frac{L(\P)}{\AN\P}\right)\text.\end{equation}

Let $\p_1$, $\ldots$, $\p_k$ be the prime ideals of $\O$ that contain the conductor $\f$, and, for each $1 \leq i \leq k$, let
\[\p_i \O_K = \P_{i, 1}^{e_{i,1}}\cdots\P_{i, l_i}^{e_{i, l_i}}\text,\]
with distinct prime ideals $\P_{i,j}$ of $\O_K$, and $e_{i,j}$, $l_i\geq 1$. Then the $\P_{i, j}$ are exactly the prime ideals of $\O_K$ dividing $\f$, that is, the elements of $\mathcal{P}_\f$.

Notice that, for every ideal $\a \mid \P_{i, 1} \cdots \P_{i, l_i}$ of $\O_K$, we have $\a \cap \O = \p_i$ if $\a \neq \O_K$, and $\a \cap \O = \O$ if $\a = \O_K$, since $\O$ is one-dimensional. As all $\p_i$, $\p_j$, $i \neq j$, are relatively prime, we get
\begin{align*}
&\sum_{\a \mid \f}\frac{\mu(\a)L_\O(\a)}{[\O : \a \cap \O]} = \prod_{i=1}^k \sum_{\a \mid \P_{i, 1}\cdots\P_{i, l_i}}\frac{\mu(\a)L_\O(\a)}{[\O : \a\cap \O]}\\
=\ &\prod_{i=1}^k\left(1 + \frac{L_\O(\P_{i,1})}{[\O : \p_i]}\sum_{\substack{J \subseteq \{1, \ldots, l_i\}\\J \neq \emptyset}}(-1)^{\left|J\right|}\right) = \prod_{i=1}^k\left(1 - \frac{L_\O(\P_{i,1})}{[\O : \p_i]}\right)\text.
\end{align*}

Thus, \eqref{corollary_goal} is equivalent to
\[\prod_{i=1}^k\left(1 - \frac{L_\O(\P_{i,1})}{[\O : \p_i]}\right) < [\O_K : \O]\prod_{i=1}^k\prod_{j=1}^{l_i}\left(1 - \frac{L(\P_{i,j})}{\AN\P_{i,j}}\right)\text.\]
Clearly, $\Pi := \prod_{i=1}^k\prod_{j=1}^{l_i}\P_{i,j}^{e_{i,j}}$ divides the conductor $\f$. Let us first assume that $\Pi$ is a proper divisor of $\f$. Then Lemma \ref{index_of_order} (with strict inequality, since $\f$ does not divide $\Pi$), \eqref{greateronehalf}, and the fact that $\AN\P \geq 4$ for all prime ideals $\P$ of $\O_K$ imply
\begin{align*}
[\O_K : \O]&\prod_{i=1}^k\prod_{j=1}^{l_i}\left(1 - \frac{L(\P_{i,j})}{\AN\P_{i,j}}\right) > \prod_{i=1}^k\frac{\AN\P_{i,1}^{e_{i,1}}}{[\O : \p_i]}\left(1 - \frac{L(\P_{i,1})}{\AN\P_{i,1}}\right)\prod_{j=2}^{l_i}\frac{\AN\P_{i,j}^{e_{i,j}}}{2}\\\geq\ &\prod_{i=1}^k\frac{\AN\P_{i,1}}{[\O : \p_i]}\left(1 - \frac{L(\P_{i,1})}{\AN\P_{i,1}}\right)2^{l_i-1} \geq \prod_{i=1}^k\left(1 - \frac{L_\O(\P_{i,1})}{[\O : \p_i]}\right)\text.
\end{align*}
For the last inequality, notice that either $\O_K/\P_{i,1} \simeq \O/\p_i$, and thus $L(\P_{i,1}) = L_\O(\P_{i,1})$, or \[\frac{\AN\P_{i,1}}{[\O : \p_i]}\left(1 - \frac{L(\P_{i,1})}{\AN\P_{i,1}}\right) > 2\cdot \frac{1}{2} = 1 \geq 1 - \frac{L_\O(\P_{i,1})}{[\O : \p_i]}\text.\]

We are left with the case where $\Pi = \f$. Then, for all $1 \leq i \leq k$, we have
\begin{equation}\label{primedividesf}l_i > 1 \text{ or } e_{i,1} > 1 \text{ or } [\O_K/\P_{i,1} : \O/\p_i] > 1\text.\end{equation}
Indeed, suppose otherwise, that is $\p_i \O_K = \P_{i,1}$ and $\O_K/\P_{i,1} \simeq \O/\p_i$, for some $i$. We put $\tilde{\O} := (\O_K)_{\P_{i,1}}$, the integral closure of the localisation $\O_{\p_i}$, $\mathfrak{m} := \p_i \O_{\p_i}$, the maximal ideal of $\O_{\p_i}$, and $\mathfrak{M} := \P_{i,1}\tilde{\O}$, the maximal ideal of $\tilde{\O}$. Then
\[[\tilde{\O} : \O_{\p_i}] = \frac{[\tilde{\O} : \mathfrak{M}][\mathfrak{M} : \mathfrak{m}]}{[\O_{\p_i} : \mathfrak{m}]} = \frac{[\O_K : \P_{i,1}][\mathfrak{M} : \mathfrak{m}]}{[\O : \p_i]} = 1\text.\]
The second equality holds because $\O_K/\P_{i,1} \simeq \tilde{\O}/\mathfrak{M}$, and $\O/\p_i \simeq \O_{\p_i}/\mathfrak{m}$. The third equality holds because $\mathfrak{M} = \P_{i,1}\tilde{\O} = \f\tilde{\O}$, whence $\mathfrak{M}$ is clearly contained in the conductor of $\O_{p_i}$ in $\tilde{\O}$. (Here we used the hypothesis $\Pi = \f$.) Therefore $\mathfrak{M} = \mathfrak{M} \cap \O_{\p_i} = \mathfrak{m}$.

Therefore, $\O_{p_i}$ is a discrete valuation ring. According to \cite[Theorem I.12.10]{Neukirch1999}, this is the case if and only if $\p_i$ does not contain $\f$. Since $\p_i$ contains $\f$, we have proved \eqref{primedividesf}. (In \cite[Section I.13]{Neukirch1999}, it is stated that \eqref{primedividesf} holds even without the requirement that $\Pi = \f$, but no proof is given.)

With Lemma \ref{index_of_order}, \eqref{greateronehalf}, and the fact that $\AN\P \geq 4$ for all prime ideals $\P$ of $\O_K$, we get
\begin{align*}
[\O_K : \O]&\prod_{i=1}^k\prod_{j=1}^{l_i}\left(1 - \frac{L(\P_{i,j})}{\AN\P_{i,j}}\right) > \prod_{i=1}^k\frac{1}{[\O : \p_i]}\prod_{j=1}^{l_i}\frac{\AN\P_{i,j}^{e_{i,j}}}{2}\\\geq\ &\prod_{i=1}^k\frac{\AN\P_{i,1}}{[\O : \p_i]}\frac{\AN\P_{i,1}^{e_{i,1}-1}}{2}2^{l_i-1} \geq \prod_{i=1}^k2^{([\O_K/\P_{i,1} : \O/\p_i] - 1) + (e_{i,1}-1) + (l_i - 1) - 1}\text.
\end{align*}
To conclude our proof, we notice that the last expression is at least $1$, by \eqref{primedividesf}. 
\section{Proof of Theorem 1}
We need to construct extensions of $K$ where we have good control over the ring of integers. This is achieved by the following two lemmata.

\begin{lemma}[({\cite[Lemma 1]{Llorente1984}})]\label{discriminant_exponent}
Let $r$ be a positive integer, and $\beta \in \O_K$, such that $g = X^r - \beta \in \O_K[X]$ is irreducible. Let $\eta$ be a root of $g$, $L = K(\eta)$, and $\D_{L|K}$ the relative discriminant of $L|K$. For every prime ideal $\P$ of $\O_K$ not dividing $\gcd(r, v_\P(\beta))$, we have
\[v_\P(\D_{L|K}) = r\cdot v_\P(r) + r - \gcd(r, v_\P(\beta))\text.\]
\end{lemma}

\begin{lemma}\label{discriminant}
Let $\omega$, $\eta \in \O_K$, such that $\omega^2 - 4 \eta$ is squarefree and relatively prime to $2$. Assume that the polynomial $h := X^2 - \omega X + \eta \in \O_K[X]$ is irreducible, and let $\alpha$ be a root of $h$. Then the ring of integers of $K(\alpha)$ is $\O_K[\alpha]$, and the relative discriminant $\D_{K(\alpha)|K}$ of $K(\alpha)$ over $K$ is the principal ideal $(\omega^2 - 4 \eta)$.
\end{lemma}

\begin{proof}
The discriminant of $\alpha$ over $K$ is 
\[d(\alpha) = \det\begin{pmatrix}1 &  (\omega + \sqrt{\omega^2 - 4\eta})/2 \\ 1 &  (\omega - \sqrt{\omega^2 - 4\eta})/2 \end{pmatrix}^2 = \omega^2 - 4\eta\text.\]
Let, say, $(\omega^2 - 4\eta) = \P_1 \cdots \P_s$, with an integer $s \geq 0$ and distinct prime ideals $\P_i$ of $\O_K$ not containing $2$. Then the relative discriminant $\D_{K(\alpha)|K}$ divides $\P_1 \cdots \P_s$.

Since $K(\alpha) = K(\sqrt{\omega^2 - 4\eta})$, Lemma \ref{discriminant_exponent} implies that $v_{\P_i}(\D_{K(\alpha)|K}) = 1$, for all $1 \leq i \leq s$, whence the relative discriminant $\D_{K(\alpha)|K}$ is the principal ideal $(\omega^2 - 4\eta) = (d(\alpha))$. This is enough to prove that the ring of integers of $K(\alpha)$ is $\O_K[\alpha]$ (see, for example, \cite[Chapter V, Theorem 30]{Zariski1975}).
\end{proof}

We may assume that $K$ satisfies the hypotheses of Proposition \ref{sqfree}, since it is enough to prove the theorem for the number field $K(\sqrt{5}) \supseteq \Q(\sqrt{5})$.

We may also assume that the field $K$ is generated by a unit of $\O_K$. If not, say $K = \Q(\beta)$, where $\beta \in \O_K$. Let $\alpha$ be a root of the polynomial $X^2 - \beta X + 1 \in \O_K[X]$. Then $\Q(\alpha) \supseteq K$, whence it is enough to prove the theorem for $\Q(\alpha)$, and $\alpha$ is a unit of the ring of integers of $\Q(\alpha)$.

Therefore, the ring generated by the units of $\O_K$ is an order. Let us call that order $\O^U$. If $\O^U = \O_K$ then there is nothing to prove, so assume from now on that $\O^U \neq \O_K$.

Choose a unit $\eta \in \O_K^*\smallsetminus K^2$. We use Proposition \ref{sqfree} to obtain elements $\omega_1$, $\ldots$, $\omega_r \in \O_K$ with 
\begin{equation}\label{omegai2}\O_K = \O^U[\omega_1, \ldots, \omega_r]\text,\end{equation}
such that 
\begin{equation}\label{omegai}\text{all $\omega_i^2 - 4\eta$ are squarefree and relatively prime to $2$ and each other.}\end{equation}
Start with 
\[\mathcal{P} := \supp(2)\text,\quad \O := \O^U\text,\]
and choose an element $\omega_1$ as in Proposition \ref{sqfree}. Then $\O^U[\omega_1]$ is an order larger than $\O^U$, whence
\[[\O_K : \O^U[\omega_1]] = \frac{[\O_K : \O^U]}{[\O^U[\omega_1] : \O^U]}\leq \frac{[\O_K : \O^U]}{2}\text.\]
Assume now that $\omega_1$, $\ldots$, $\omega_{i-1}$ have been chosen. If $\O^U[\omega_1, \ldots, \omega_{i-1}] = \O_K$ then stop, otherwise put 
\[\mathcal{P} := \supp(2) \cup \bigcup_{j=1}^{i-1} \supp(\omega_j^2-4\eta)\text,\quad \O := \O^U[\omega_1, \ldots, \omega_{i-1}]\text.\]
Let $\omega_i$ be an element as in Proposition \ref{sqfree}. Then 
\[[\O_K : \O^U[\omega_1, \ldots, \omega_i]] \leq [\O_K : \O^U[\omega_1, \ldots, \omega_{i-1}]]/2 \leq [\O_K : \O^U]/2^i\text.\]
Therefore, the above process stops after $r \leq \log_2([\O_K : \O^U])$ steps, with elements $\omega_1$, $\ldots$, $\omega_r \in \O_K \smallsetminus \O^U$, such that $\O_K = \O^U[\omega_1, \ldots, \omega_r]$. Conditions \eqref{omegai} hold by our construction.

For $1 \leq i \leq r$, let $\alpha_i$ be a root of the polynomial $X^2 - \omega_i X + \eta \in \O_K[X]$. Then $\alpha_i$ is a unit in the ring of integers of $K(\alpha_i)$. Moreover, $\alpha_i \notin K$, since otherwise $\alpha_i \in \O_K^*$, and $\omega_i = \alpha_i + \eta\alpha_i^{-1} \in \O^U$, a contradiction. By Lemma \ref{discriminant}, the ring of integers of $K(\alpha_i)$ is $\O_K[\alpha_i]$, and the relative discriminant $\D_{K(\alpha_i)|K}$ of $K(\alpha_i)$ over $K$ is the principal ideal $(\omega_i^2 - 4\eta)$.

We use the following well-known fact (for a proof, see \cite[Theorem I.2.11]{Neukirch1999}): 

\begin{lemma}\label{galois_comp}
Let $L|K$ and $L'|K$ be two Galois extensions of $K$ such that
\begin{enumerate}
\item $L \cap L' = K$,
\item $L$ has a relative integral basis $\{\beta_1, \ldots, \beta_l\}$ over $K$,
\item $L'$ has a relative integral basis $\{\beta_1', \ldots, \beta_{l'}'\}$ over $K$, and
\item the relative discriminants $\D_{L|K}$ and $\D_{L'|K}$ are relatively prime\text.
\end{enumerate}
Then the compositum $LL'$ has a relative integral basis over $K$ consisting of all products $\beta_i \beta_j'$, and the relative discriminant of $LL'|K$ is
\begin{equation*}\D_{LL'|K} = \D_{L|K}^{[L':K]}\D_{L'|K}^{[L:K]}\text.\end{equation*}
\end{lemma}

Consider the extension fields $L_i := K(\alpha_1, \ldots, \alpha_i)$ of $K$. We claim that $L_i$ has an integral basis over $K$ consisting of (not necessarily all) products of the form  
\[\prod_{j \in J}\alpha_j, \quad \text{ for }J \subseteq \{1, \ldots, i\}\text,\]
and that the relative discriminant $\D_{L_i | K}$ is relatively prime to all relative discriminants $\D_{K(\alpha_j)|K}$, for $i < j \leq r$.

With \eqref{omegai}, this claim clearly holds for $L_1 = K(\alpha_1)$. If the claim holds for $L_{i-1}$, and $\alpha_i \in L_{i-1}$, then it holds for $L_i = L_{i-1}$ as well. If $K(\alpha_i) \not\subseteq L_{i-1}$ then the extensions $L_{i-1}|K$ and $K(\alpha_i)|K$ satisfy all requirements of Lemma \ref{galois_comp}, whence the claim holds as well for $L_i = L_{i-1}K(\alpha_i)$.

Now put $L := L_r$. Then the ring of integers of $L$ is $\O_L = \O_K[\alpha_1, \ldots, \alpha_r]$. With \eqref{omegai2} and $\omega_i = \alpha_i + \eta\alpha_i^{-1}$, we get
\[\O_L = \O^U[\omega_1, \ldots, \omega_r, \alpha_1, \ldots, \alpha_r] = \O^U[\alpha_1, \alpha_1^{-1}, \ldots, \alpha_r, \alpha_r^{-1}]\text,\]
and the latter ring is generated by units of $\O_L$. 

\subsection*{Acknowledgements}
I would like to thank Martin Widmer for many helpful comments and discussions, in particular about Lemma \ref{counting_integers} and the linear transformation $T$ that occurs there. The idea of using such transformations stems from an upcoming paper by Widmer.

\bibliographystyle{plain}
\bibliography{extension_problem}
\end{document}